\documentclass[11pt,a4paper]{article}

\usepackage{amsmath,amssymb}
\usepackage{mathrsfs,amsfonts}
\usepackage{bm}
\usepackage{graphicx}
\usepackage{epsfig}
\usepackage{indentfirst}
\usepackage{extarrows}
\usepackage{amsthm}

\begin{document}


\title{\textbf{The existence and nonexistence of global solutions for a semilinear heat equation on graphs}}
\author{Yong Lin \quad Yiting Wu }

\date{}
\maketitle

\vspace{-12pt}


\begin{minipage}{145mm}
\noindent{\small\textbf{Abstract} \; Let $G=(V,E)$ be a finite or locally finite connected weighted graph, $\Delta$ be the usual graph Laplacian. Using heat kernel estimate, we prove the existence and nonexistence of global solutions for the following semilinear heat equation on $G$
\begin{equation*}
\left\{
\begin{array}{lc}
u_t=\Delta u + u^{1+\alpha} &\, \text{in $(0,+\infty)\times V$,}\\
u(0,x)=a(x) &\, \text{in $V$.}
\end{array}
\right.
\end{equation*}
We conclude that, for a graph satisfying curvature dimension condition
$CDE'(n,0)$ and $V(x,r)\simeq r^m$, if $0<m\alpha<2$, then the non-negative solution $u$ is not global, and if $m\alpha>2$, then there is a non-negative global solution $u$ provided that the initial value is small enough. In particular, these results are true on lattice $\mathbb{Z}^m$.
}
\end{minipage}

\vspace{12pt}

\section{Introduction}

The existence or nonexistence of global solutions to a simple system
\begin{equation*}
\tag{1.1}
\left\{
\begin{array}{lc}
u_t=\Delta u + u^{1+\alpha} \quad (t>0, x\in \mathbb{R}^m),\\
u(0,x)=a(x) \quad (x\in \mathbb{R}^m),
\end{array}
\right.
\end{equation*}
has been extensively studied since 1960s. One of the most important results about it is from Fujita \cite{Fujita}.
Fujita claimed that, if $0<m\alpha<2$, then there does not exist a non-negative global solution for any non-trivial non-negative initial data. On the other hand, if $m\alpha>2$, then there exists a global solution for a sufficiently small initial data. It is clear that Fujita's result does not include the critical exponent
$\alpha=\frac{2}{m}$. The nonexistence of global solutions for
critical exponent was proved in \cite{Hayakawa,Kobayashi}.

Recently, the study of equations on graphs has attracted attention from many researchers in various fields
(see \cite{CLC,gly1,gly2,gly3,LCY,XXM} and references therein).
Grigoryan et al. \cite{gly1,gly2,gly3} established existence results for Yamabe type equations and some nonlinear elliptic equations on graphs.
The solutions of heat equation and its variations on graphs have also been investigated by many authors due to its wide range of applications ranging from modelling energy flows through a network to processing image \cite{Network,Image}.
Chung et al. \cite{CLC} considered the extinction and positivity of the solutions of the Dirichlet boundary value problem for $u_t=\Delta u - u^q$ with $q>0$ on a network.

In \cite{XXM}, Xin et al. studied the blow-up properties of the Dirichlet boundary value problem for $u_t=\Delta u + u^q$ with $q>0$ on a finite graph. They concluded that if $q\leq 1$, every solution is global, and if $q> 1$ and under some suitable conditions, the nontrivial solutions blow up in finite time.
Different from \cite{XXM}, in this paper we consider the sufficient conditions for existence or nonexistence of global solutions of the Cauchy problem for $u_t=\Delta u + u^{1+\alpha}$ with $\alpha >0$ on a finite or locally finite graph.

From another perspective, the problem discussed in this paper can be regarded as a discrete analogue of the problem (1.1), that is,
\begin{equation}
\tag{1.2}
\left\{
\begin{array}{lc}
u_t=\Delta u + u^{1+\alpha} &\, \text{in $(0,+\infty)\times V$,}\\
u(0,x)=a(x) &\, \text{in $V$.}
\end{array}
\right.
\end{equation}

Motivated by \cite{Fujita}, we find that the key technical point of proving the existence of global solutions is the estimate of heat kernel. In \cite{BHLLMY},
Bauer et al. obtained the Gaussian upper bound for a graph satisfying $CDE(n,0)$, furthermore, in \cite{HLLY}, Horn et al. derived the Gaussian lower bound for a graph satisfying $CDE'(n,0)$. In addition, Lin et al. \cite{LW} used the volume growth condition to obtain a weaker on-diagonal lower estimate of heat kernel on graphs for large time.
Using these heat kernel estimates, we can prove the existence and nonexistence of global solutions for problem (1.2) on finite or locally finite graphs.

The results of Fujita \cite{Fujita} reveal that the dimension of the space and the degree of non-linearity of the equation have a combined effect on deciding whether a solution of (1.1) exists
globally in Euclidean space. It is worth noting that the main results of this paper exactly show that, for a graph satisfying $CDE'(n,0)$ and $V(x,r)\simeq r^m$, the behaviors of the solutions for problem (1.2) strongly depend on $m$ and $\alpha$. In particular, for lattice $\mathbb{Z}^m$, we have similar results of Fujita \cite{Fujita} on Euclidean space $\mathbb{R}^m$.

The rest of the paper is organized as follows. In Section 2, we introduce some concepts, notations and known results which are essential to prove the main results of this paper.
In Section 3, we formally state our main results. In Sections 4 and 5, we respectively prove the
nonexistence and existence of global solutions for problem (1.2).
In Section 6, we study the behavior of the solutions for problem (1.2) under the curvature condition $CDE'$.
In Section 7, we give an example to explain our conclusions intuitively. Meanwhile,
we also provide a numerical experiment to demonstrate the example.

\section{Preliminaries}

Throughout the paper, we assume that $G=(V,E)$ is a finite or locally finite connected graph and contain neither loops nor multiple edges, where $V$ denotes the vertex set and $E$ denotes the edge set. We write $y\sim x$ if $y$ is adjacent to $x$, or equivalently $\overline{xy}\in E$. For each vertex $x$, its degree is defined by
$$\deg(x)=\#\{y\in V:y\sim x\}.$$

We allow the edges on the graph to be weighted. Weights are given by a function $\omega: V \times V\rightarrow[0,\infty)$, that is, the edge $\overline{xy}$ has weight $\omega_{xy}\geq0$ and $\omega_{xy}=\omega_{yx}$.
Furthermore, let $\mu: V\rightarrow \mathbb{R}^+$ be a positive finite measure on the vertices of the $G$. In this paper,
all the graphs in our concern are assumed to satisfy
$$D_\omega:=\frac{\mu_{\max}}{\omega_{\min}}<\infty$$
and
$$D_\mu:=\max_{x\in V}\frac{m(x)}{\mu(x)}<\infty,$$
where $\omega_{\min}:=\inf_{\overline{xy}\in E}\omega_{xy}>0$ and
$m(x):=\sum_{y\sim x}\omega_{xy}.$

\subsection{Laplace operators on graphs}

Let $C(V)$ be the set of real functions on $V$. For any $1\leq p<\infty$, we denote  by
$$\ell^p(V,\mu)=\left \{f\in C(V):\sum_{x \in V} \mu(x)|f(x)|^p<\infty\right \}$$
the set of $\ell^p$ integrable functions on $V$ with respect to
the measure $\mu$.  For $p=\infty$, let
$$\ell^{\infty}(V,\mu)=\left\{f\in C(V):\sup_{x\in V}|f(x)|<\infty\right\}.$$

For any function $f\in C(V)$, the $\mu$-Laplacian $\Delta$ of $f$ is defined by
$$\Delta f(x)=\frac{1}{\mu(x)}\sum_{y\sim x}\omega_{xy}(f(y)-f(x)),$$
it can be checked that $D_\mu<\infty$ is equivalent to the $\mu$-Laplacian $\Delta$ being bounded on $\ell^p(V,\mu)$ for
all $p\in [1,\infty]$ (see \cite{HAESELER}).
The special cases of $\mu$-Laplacian operators are the cases where $\mu\equiv 1$, which is the standard graph Laplacian, and the case where $\mu(x)=\sum_{y\sim x}\omega_{xy}=\deg(x)$, which yields the normalized graph Laplacian.

The gradient form $\Gamma$ associated with a $\mu$-Laplacian is defined by
$$\Gamma(f,g)(x)=\frac{1}{2\mu(x)}\sum_{y\sim x}\omega_{xy}(f(y)-f(x))(g(y)-g(x)).$$
We write $\Gamma(f)=\Gamma(f,f)$.

The iterated gradient form $\Gamma_2$ is defined by
$$2\Gamma_2(f,g)=\Delta\Gamma(f,g)-\Gamma(f,\Delta g)-\Gamma(\Delta f,g).$$
We write $\Gamma_2(f)=\Gamma_2(f,f).$

Besides, the integration of a function $f\in C(V)$ is defined by
$$\int_V fd\mu=\sum_{x\in V}\mu(x)f(x).$$

The connected graph can be endowed with its graph distance $d(x,y)$, i.e. the smallest number of edges of a path between two vertices $x$ and $y$, then we define balls $B(x,r)=\{y\in V:d(x,y)\leq r\}$ for any $r\geq0$. The volume of a subset $A$ of $V$ can be written as $V(A)$ and $V(A)=\sum_{x\in A}\mu(x)$, for convenience, we usually abbreviate $V\big(B(x,r)\big)$ by $V(x,r)$. In addition, a graph $G$ satisfies a uniform volume growth of degree $m$, if for all $x\in V$, $r\geq 0$,
$$V(x,r)\simeq r^m,$$
that is, there exists a constant $c'\geq 1$, such that $\frac{1}{c'}r^m\leq V(x,r) \leq c'r^m$.

\subsection{The heat kernel on graphs}

We say that a function $p:(0,+\infty)\times V \times V\rightarrow \mathbb{R}$ is a fundamental solution of the heat
equation $u_t=\Delta u$ on $G=(V,E)$, if for any bounded initial condition $u_0:V\rightarrow \mathbb{R}$, the function
$$u(t,x)=\sum_{y\in V}p(t,x,y)u_0(y) \quad (t>0, \, x\in V)$$
is differentiable in $t$ and satisfies the heat equation, and for any $x\in V$,
$\lim\limits_{t\rightarrow0^+}u(t,x)=u_0(x)$ holds.

For completeness, we recall some important properties of the heat kernel $p(t,x,y)$, as follows:
\newtheorem{proposition}{\textbf{Proposition}}[section]
\begin{proposition}[see \cite{HLLY,RW}]
\textnormal{For $t,s>0$ and any $x,y\in V$, we have\\
(i) \, $p(t,x,y)=p(t,y,x)$,\\
(ii) \, $p(t,x,y)> 0$,\\
(iii) \, $\sum_{y\in V}\mu(y)p(t,x,y)\leq 1$,\\
(iv) \, $\partial_t p(t,x,y)=\Delta_xp(t,x,y)=\Delta_yp(t,x,y)$,\\
(v) \, $\sum_{z\in V}\mu(z)p(t,x,z)p(s,z,y)=p(t+s,x,y)$.}
\end{proposition}

In \cite{BHLLMY}, Bauer et al. introduced two slightly different curvature conditions which are called $CDE$ and $CDE'$.  Let us now recall the two definitions.

\newtheorem{definition}{\textbf{Definition}}[section]
\begin{definition}
\textnormal{A graph $G$ satisfies the exponential curvature dimension inequality $CDE(x,n,K)$, if for any positive function $f:V\rightarrow \mathbb{R}^+$ such that $\Delta f(x)<0$, we have
$$\Gamma_2(f)(x)-\Gamma\Bigg(f,\frac{\Gamma(f)}{f}\Bigg)(x)\geq \frac{1}{n}(\Delta f)(x)^2+K\Gamma(f)(x),$$
we say that $CDE(n,K)$ is satisfied if $CDE(x,n,K)$ is satisfied for all $x\in V$.}
\end{definition}

\begin{definition}
\textnormal{A graph $G$ satisfies the exponential curvature dimension inequality $CDE'(x,n,K)$, if for any positive function $f:V\rightarrow \mathbb{R}^+$, we have
$$\Gamma_2(f)(x)-\Gamma\Bigg(f,\frac{\Gamma(f)}{f}\Bigg)(x)\geq \frac{1}{n}f(x)^2(\Delta \log f)(x)^2+K\Gamma(f)(x),$$
we say that $CDE'(n,K)$ is satisfied if $CDE'(x,n,K)$ is satisfied for all $x\in V$.}
\end{definition}

Bauer et al. \cite{BHLLMY} established
a discrete analogue of the Li-Yau inequality and derived a heat kernel estimate under the condition of $CDE(n,0)$, as follows:

\begin{proposition}[see \cite{BHLLMY}]
\textnormal{Suppose $G$ satisfies $CDE(n,0)$, then there exists a positive constant $C_1$ such that, for any $x,y \in V$ and $t>0$,
\begin{equation*}
\tag{2.1}
p(t,x,y)\leq \frac{C_1}{V(x,\sqrt{t})}.
\end{equation*}
Furthermore, for any $t>1$, there exists constants $C_2$ and $C_3$ such that
\begin{equation*}
\tag{2.2}
p(t,x,y)\geq C_2\frac{1}{t^n}\exp{\Bigg(-C_3\frac{d^2(x,y)}{t-1}\Bigg)}.
\end{equation*}}
\end{proposition}

Although the upper bound in the result of Bauer et al. \cite{BHLLMY} is formulated with Gaussian form, the lower bound is not quite Gaussian form and is dependent on the parameter $n$. Based on this, Horn et al. \cite{HLLY}
used $CDE'$ to imply volume doubling and derived the Gaussian type on-diagonal lower bound.
Here, we transcribe a relevant result of \cite{HLLY} as follows:

\begin{proposition}[see \cite{HLLY}]
\textnormal{Suppose $G$ satisfies $CDE'(n,0)$, then for any $x\in V$ and $t>\frac{1}{2}$,
\begin{equation*}
\tag{2.3}
p(2t^2,x,x)\geq \frac{C}{V(x,t)},
\end{equation*}
where $C>0$.}
\end{proposition}

Without the use of the curvature condition $CDE'$, Lin et al. \cite{LW} only utilized the volume growth condition to obtain a on-diagonal lower estimate of heat kernel on graphs for large time, which is enough to prove the nonexistence of global solution of (3.1) stated in Section 3. We recall it bellow.

\begin{proposition}[see \cite{LW}]
\textnormal{Assume that, for all $x\in V$ and $r\geq r_0$,
$$V(x,r)\leq c_0 r^m,$$
where $r_0,c_0,m$ are some positive constants. Then, for all large enough $t$,
\begin{equation*}
\tag{2.4}
p(t,x,x)\geq \frac{1}{4V(x,C_0 t\log t)},
\end{equation*}
where $C_0>2D_\mu e$.}
\end{proposition}

\section{Main results}

In this paper, we study whether or not there exist global solutions to the initial value problem for the semilinear heat equation
\begin{equation}
\tag{3.1}
\left\{
\begin{array}{lc}
u_t=\Delta u + u^{1+\alpha} &\, \text{in $(0,+\infty)\times V$,}\\
u(0,x)=a(x) &\, \text{in $V$,}
\end{array}
\right.
\end{equation}
where $\alpha$ is a positive parameter, $a(x)$ is bounded, non-negative and not trivial in $V$. Without loss of generality, we can assume that $a(e)>0$ with $e\in V$. Throughout the present paper we shall only deal with non-negative solutions so that there is no ambiguity in the meaning of $u^{1+\alpha}$.
We shall also fix the vertex $e$.

For convenience, we state relevant definitions firstly.

\begin{definition}
\textnormal{Assume that $T>0$, a non-negative function $u=u(t,x)$ satisfying (3.1) in $[0,T]\times V$ is called a solution of (3.1) in $[0,T]$, if $u$ is bounded and continuous with respect to $t$. Furthermore, a solution $u$ of (3.1) in $[0,+\infty)$ is a function whose restriction to $[0,T]\times V$ is a solution of (3.1) in $[0,T]$ for any $T>0$. A solution $u$ of (3.1) in $[0,+\infty)$ is also called a global solution of (3.1) in $[0,+\infty)$.}
\end{definition}

\begin{definition}
\textnormal{$\mathcal{F}[0,+\infty)$ is the set of all non-negative continuous(with respect to $t$) functions $u=u(t,x)$ defined in $[0,+\infty)\times V$ satisfying
$$0\leq u(t,x) \leq Mp(t+\gamma,e,x)$$
with some constants $M>0$ and $\gamma>0$. Furthermore, if $u$ is a solution of (3.1) in $[0,+\infty)$ and $u\in \mathcal{F}[0,+\infty)$, then $u$ is called a global solution of (3.1) in $\mathcal{F}[0,+\infty)$.}
\end{definition}


Our main results are stated in the following theorems.

\newtheorem{theorem}{\textbf{Theorem}}[section]
\begin{theorem}
\textnormal{Assume that, for all $x\in V$ and $r\geq r_0$, the volume growth $V(x,r)\leq c_0 r^m$ holds, where $r_0,c_0,m$ are some positive constants. If $0<m\alpha<1$, then there is no non-negative global solution of (3.1) in $[0,+\infty)$ for any bounded, non-negative and non-trivial initial value.}
\end{theorem}

\begin{theorem}
\textnormal{Assume that $G$ satisfies $CDE(n,0)$ and $V(x,r)\geq c_1 r^m$ with some positive constants $c_1$ and $m$. Suppose for any $\gamma>0$, there exists a positive number $\delta$ such that $0\leq a(x)\leq \delta p(\gamma,e,x)$ in $V$. If $m\alpha>2$, then (3.1) has a global solution $u=u(t,x)$ in $\mathcal{F}[0,+\infty)$, which satisfies $0\leq u(t,x)\leq Mp(t+\gamma,e,x)$, for any $(t,x)\in [0,+\infty)\times V$ and some positive constants $M(\delta)$.}
\end{theorem}

\newtheorem{corollary}{\textbf{Corollary}}[section]
\begin{corollary}
\textnormal{Suppose $G$ satisfies $CDE'(n,0)$ and $V(x,r)\simeq r^m$ for some $m>0$.\\
(i) If $0<m\alpha<2$, then there is no non-negative global solution of (3.1) in $[0,+\infty)$ for any bounded, non-negative and non-trivial initial value.\\
(ii) If $m\alpha>2$, then there exists a global solution of (3.1) in $\mathcal{F}[0,+\infty)$ for a sufficiently small initial value.
}
\end{corollary}

\section{Proof of Theorem 3.1}


We first introduce a lemma which will be used in the proof of Theorem 3.1.

\newtheorem{lemma}{\textbf{Lemma}}[section]
\begin{lemma}
\textnormal{Let $T>0$, if $u=u(t,x)$ is a non-negative solution of (3.1) in $[0,T]$, then we have
$$J_0^{-\alpha}-u(t,e)^{-\alpha}\geq\alpha t \quad (0< t\leq T),$$
where
$$J_0=J_0(t)=\sum_{x\in V} \mu(x)p(t,e,x)a(x).$$}
\end{lemma}


\begin{proof}[Proof]
Let $\varepsilon$ be a positive constant and for any fixed $t\in(0,T]$, we put
\begin{equation*}
v_\varepsilon(s,x)=p(t-s+\varepsilon,e,x) \quad (0\leq s\leq t, \; x\in V)
\end{equation*}
and
\begin{equation*}
J_\varepsilon(s)=\sum_{x\in V}\mu(x)v_\varepsilon(s,x)u(s,x) \quad (0\leq s\leq t).
\end{equation*}

(i) We prove that $J_\varepsilon$ is positive for all $s\in[0,t]$.

Since $a(e)=u(0,e)>0$ and $u(s,e)$ is non-negative in $(0,t]$, for all $0\leq s\leq t$, it follows that
\begin{equation*}
\tag{4.1}
\frac{\partial u}{\partial s}(s,e)-\Delta u(s,e)\geq 0.
\end{equation*}

Note that
\begin{equation*}
\begin{split}
\Delta u(s,e)&=\frac{1}{\mu(e)}\sum_{y\sim e}\omega_{ey}\big(u(s,y)-u(s,e)\big)\\
&\geq -\frac{1}{\mu(e)}\sum_{y\sim e}\omega_{ey}u(s,e)\\
&\geq -D_\mu u(s,e),
\end{split}
\end{equation*}
then the inequality (4.1) gives
\begin{equation*}
\frac{\partial u}{\partial s}(s,e)\geq -D_\mu u(s,e),
\end{equation*}
which implies
\begin{equation*}
u(s,e)\geq u(0,e)\exp(-D_\mu s)>0, \quad s\in [0,t].
\end{equation*}

Hence, for all $0\leq s \leq t$, we have
$$\sum_{x\in V}u(s,x)>0.$$

In view of the fact that $v_\varepsilon(s,x)$ is positive in $[0,t]\times V$,
we obtain $J_\varepsilon(s)>0$ in $[0,t]$.

(ii) We prove that $J_\varepsilon$ is differentiable with respect to $s$ and satisfies the following equation
\begin{equation*}
\frac{d}{ds}J_\varepsilon(s)=\sum_{x\in V} \mu(x)v_\varepsilon (s,x)u(s,x)^{1+\alpha}.
\end{equation*}

\textit{Case 1.}
We consider the case where $G$ is a finite connected graph.

Since $\omega_{xy}=\omega_{yx}$, according to the definition of $\Delta$, for any function $f,g\in C(V)$, we have
\begin{equation*}
\tag{4.2}
\sum_{x\in V}\mu(x)\Delta f(x) g(x)=\sum_{x\in V}\mu(x)f(x)\Delta g(x).
\end{equation*}

From the property of the heat kernel, we know that
\begin{equation*}
\frac{\partial}{\partial s}v_\varepsilon=-\Delta v_\varepsilon.
\end{equation*}

Thus
\begin{equation*}
\tag{4.3}
\begin{split}
\frac{d}{ds}J_\varepsilon(s)=&
\sum_{x\in V} \left( \mu(x)\frac{\partial}{\partial s}v_\varepsilon (s,x)u(s,x)+\mu(x)v_\varepsilon (s,x)\frac{\partial}{\partial s}u(s,x) \right)\\
=&\sum_{x\in V} \left( - \mu(x)\Delta v_\varepsilon (s,x)u(s,x)+\mu(x)v_\varepsilon (s,x)
\Big( \Delta u(s,x)+ u(s,x)^{1+\alpha} \Big) \right)\\
=&-\sum_{x\in V} \mu(x)\Delta v_\varepsilon (s,x)u(s,x)+\sum_{x\in V} \mu(x)v_\varepsilon (s,x)\Delta u(s,x)\\
&+\sum_{x\in V} \mu(x)v_\varepsilon (s,x)u(s,x)^{1+\alpha} \\
=&\sum_{x\in V} \mu(x)v_\varepsilon (s,x)u(s,x)^{1+\alpha}.
\end{split}
\end{equation*}

\textit{Case 2.}
We consider the case where $G$ is a locally finite connected graph.

Firstly, we claim that $J_\varepsilon$ exists if $G$ is locally finite.

Since $u$ is bounded, we can assume there exists a constant $A>0$ such that for any $(s,x)\in [0,t]\times V$,
$$|u(s,x)|\leq A.$$

Hence, from the property of the heat kernel, we have
$$J_\varepsilon=\Bigg |\sum_{x\in V}\mu(x)v_\varepsilon (s,x)u(s,x)\Bigg |\leq A\sum_{x\in V}\mu(x)v_\varepsilon (s,x)
\leq A<\infty.$$

Secondly,
we observe that if $G$ is locally finite,
the exchange between summation and derivation in the first step of (4.3) is because $J_\varepsilon(s)$ and $\frac{d}{d s}J_\varepsilon (s)$ both are uniformly convergent.

Indeed, when $\Delta$ is a bounded operator, we have
\begin{equation*}
\tag{4.4}
P_t u(x)=e^{t\Delta}u(x)=\sum_{k=0}^{+\infty}\frac{t^k \Delta^k}{k!}u(x)=\sum_{y\in V}\mu(y)p(t,x,y)u(y),
\end{equation*}
furthermore, we can prove that the summation (4.4) has a nice convergency when $u(x)$ is a bounded function.
The details are as follows:

Assuming that $|u(x)|\leq A$ in $V$, then
\begin{equation*}
|\Delta u(x)|=\Bigg |\frac{1}{\mu(x)}\sum_{y\sim x}\omega_{xy}\big(u(y)-u(x)\big)\Bigg |\leq 2D_\mu A.
\end{equation*}

By iteration, we obtain for any $k\in\mathbb{N}$ and $x\in V$,
$$\big |\Delta^ku(x)\big |\leq2^kD_\mu^kA.$$

Thus for any $t\in (0,T]$ and $x\in V$,
$$\Bigg | \frac{t^k\Delta^k}{k!}u(x) \Bigg | \leq \Bigg | \frac{T^k\Delta^k}{k!}u(x) \Bigg |
\leq \frac{T^k}{k!}2^kD_\mu^kA.$$

In view of
$$\sum_{k=0}^{+\infty}\frac{T^k}{k!}2^kD_\mu^kA=Ae^{2D_\mu T}<\infty,$$
which shows $\sum_{y\in V}\mu(y)p(t,x,y)u(y)$ converges uniformly on $(0,T]$, when $u(x)$ is bounded in $V$.

Since $u(s,x)$ and $u(s,x)^{1+\alpha}$ both are bounded, we can obtain that $J_\varepsilon (s)$ and $\frac{d}{ds}J_\varepsilon(s)$ converge uniformly on $[0,t]$.

Thirdly, we notice that if $G$ is locally finite, (4.2) may not always hold, but for any bounded function
$u$, it satisfies
\begin{equation*}
\tag{4.5}
\sum_{y\in V} \mu(y)\Delta p(t,x,y)u(y)=\sum_{y\in V} \mu(y)p(t,x,y)\Delta u(y).
\end{equation*}

A direct computation yields
\begin{equation*}
\begin{split}
\sum_{y\in V} \mu(y)\Delta p(t,x,y)u(y)=&
\sum_{y\in V}\sum_{z\in V}\omega_{yz}\big(p(t,x,z)u(y)-p(t,x,y)u(y)\big)\\
=&\sum_{y\in V}\sum_{z\in V}\omega_{yz}p(t,x,z)u(y)-\sum_{y\in V}\sum_{z\in V}\omega_{yz}p(t,x,y)u(y)\\
=&\sum_{z\in V}\sum_{y\in V}\omega_{yz}p(t,x,y)u(z)-\sum_{y\in V}\sum_{z\in V}\omega_{yz}p(t,x,y)u(y)\\
=&\sum_{y\in V}\sum_{z\in V}\omega_{yz}p(t,x,y)u(z)-\sum_{y\in V}\sum_{z\in V}\omega_{yz}p(t,x,y)u(y)\\
=&\sum_{y\in V} \mu(y)p(t,x,y)\Delta u(y).
\end{split}
\end{equation*}

Note that the summation can be exchanged, since
\begin{equation*}
\begin{split}
\sum_{y\in V}\sum_{z\in V}\big|\omega_{yz}p(t,x,y)u(z)\big|
&\leq \sum_{y\in V}\mu(y)p(t,x,y)\left(\sum_{z\in V}\frac{\omega_{yz}}{\mu(y)}\left|u(z)\right|\right)\\
&\leq D_\mu A.
\end{split}
\end{equation*}

Finally, we state that if $G$ is locally finite, the interchanges of sum in the third step of (4.3) are again because
of the convergence of the sums.

Noting that $|\Delta u(s,x)| \leq 2D_\mu A$, \, $|u(s,x)^{1+\alpha}|\leq A^{1+\alpha}$,
and
$$\sum_{x\in V} \mu(x)\Delta v_\varepsilon (s,x)u(s,x)=\sum_{x\in V} \mu(x)v_\varepsilon (s,x)\Delta u(s,x),$$
for any $(s,x)\in [0,t]\times V$, we deduce that
$\sum_{x\in V} \mu(x)v_\varepsilon(s,x)\Delta u(s,x)$,\, $\sum_{x\in V} \mu(x)v_\varepsilon(s,x)u(s,x)^{1+\alpha}$
and $\sum_{x\in V} \mu(x)\Delta v_\varepsilon (s,x)u(s,x)$ all are convergent.\\

(iii) Since $v_\varepsilon > 0$ and
\begin{equation*}
\sum_{x\in V}\mu(x)v_\varepsilon(s,x)\leq 1,
\end{equation*}
using the Jensen's inequality to $x^{1+\alpha}$ $(\alpha>0)$ and owing to its convexity, we obtain
\begin{equation*}
\frac{\sum_{x\in V} \mu(x)v_\varepsilon (s,x)u(s,x)^{1+\alpha}}{\sum_{x\in V}\mu(x)v_\varepsilon (s,x)}\geq \left(\frac{\sum_{x\in V} \mu(x)v_\varepsilon (s,x)u(s,x)}{\sum_{x\in V}\mu(x)v_\varepsilon (s,x)}\right)^{1+\alpha},
\end{equation*}
that is,
\begin{equation*}
\begin{split}
&\left(\sum_{x\in V} \mu(x)v_\varepsilon (s,x)u(s,x)\right)^{1+\alpha}\\
\leq&\left(\sum_{x\in V} \mu(x)v_\varepsilon (s,x)u(s,x)^{1+\alpha}\right)
\left(\sum_{x\in V} \mu(x)v_\varepsilon (s,x)\right)^\alpha\\
\leq&\sum_{x\in V} \mu(x)v_\varepsilon (s,x)u(s,x)^{1+\alpha}.
\end{split}
\end{equation*}

It follows that
$$\frac{d}{ds}J_\varepsilon \geq J_\varepsilon^{1+\alpha}.$$

Using the Mean-value theorem, we have
\begin{equation*}
\tag{4.6}
J_\varepsilon(0)^{-\alpha}-J_\varepsilon(t)^{-\alpha}\geq\alpha t.
\end{equation*}

According to (4.4), we can assert that for any bounded function $u$,
\begin{equation*}
\lim_{t\rightarrow 0^+}P_t u(x)=\lim_{t\rightarrow 0^+}\sum_{y\in V} \mu(y)p(t,x,y)u(y)=u(x),
\end{equation*}
from which we will get
\begin{equation*}
\tag{4.7}
J_\varepsilon(t)\rightarrow u(t,e) \quad\quad (\varepsilon\rightarrow 0^+).
\end{equation*}

Moreover, it is not difficult to find that
\begin{equation*}
\tag{4.8}
J_\varepsilon(0)\rightarrow J_0 \quad\quad (\varepsilon\rightarrow 0^+).
\end{equation*}

In fact, if $G$ is a finite connected graph, the (4.8) is obvious.
If $G$ is a locally finite connected graph, because of the uniform convergence of $J_\varepsilon(s)$,
we can exchange limitation with summation and obtain (4.8).

Combining (4.7) and (4.8) into (4.6), for any $t\in (0,T]$, we have
$$J_0^{-\alpha}-u(t,e)^{-\alpha}\geq\alpha t.$$

This completes the proof of Lemma 4.1.
\end{proof}


\begin{proof}[Proof of Theorem 3.1.]
Based on the above Lemmas, we prove Theorem 3.1 by contradiction.

Suppose that there exists a non-negative global solution $u=u(t,x)$ of (3.1) in $[0,+\infty)$, according to Lemma 4.1,
we have for any $t>0$,
$$J_0^{-\alpha}\geq u(t,e)^{-\alpha}+\alpha t \geq \alpha t.$$

Since $V(x,r)\leq c_0 r^m$, $(r\geq r_0)$, from Proposition 2.4, we have for all large enough $t$,
\begin{equation*}
p(t,e,e)\geq \frac{1}{4c_0C_0^m}\left(t\log t\right)^{-m} \quad\, (C_0>2D_\mu e).
\end{equation*}

Hence, for all sufficiently large $t$,
\begin{equation*}
\begin{split}
J_0&=\sum_{x\in V}\mu(x)p(t,e,x)a(x)\\
&\geq \mu(e)a(e)p(t,e,e)\\
&\geq \overline{C}\left(t\log t\right)^{-m},
\end{split}
\end{equation*}
where $\overline{C}=\frac{\mu(e)a(e)}{4c_0C_0^m}>0$ and $C_0>2D_\mu e$.\\

Combining $J_0^{-\alpha}\geq \alpha t$ and $J_0\geq \overline{C}\left(t\log t\right)^{-m}$, for all large enough $t$,
we get
\begin{equation*}
\tag{4.9}
\left(t\log t\right)^{m\alpha}\geq \alpha\overline{C}^\alpha t.
\end{equation*}

However, if $0<m\alpha<1$, we will get a contradiction for large enough $t$.

This completes the proof of Theorem 3.1.
\end{proof}

\section{Proof of Theorem 3.2}

Before proving Theorem 3.2, we consider the following integral equations associated with (3.1) and obtain its solution $u(t,x)$ in $\mathcal{F}(0,+\infty)$.
\begin{equation}
\tag{5.1}
\left\{
\begin{split}
&u(t,x)=u_0(t,x)+\int_0^t\sum\limits_{y\in V}\mu(y)p(t-s,x,y)u(s,y)^{1+\alpha} ds &\quad \text{in $(0,+\infty)\times V$},\\
&u_0(t,x)=\sum\limits_{y\in V}\mu(y)p(t,x,y)a(y) &\quad \text{in $(0,+\infty)\times V$,}
\end{split}
\right.
\end{equation}
where $\alpha>0$, $a(y)$ is bounded, non-negative, not trivial and satisfying $0\leq a(y)\leq\delta p(\gamma,e,y)$ in $V$ with some constants $\gamma>0$ and $\delta>0$. We fix $\gamma$ here and will determine $\delta$ later.

For any function $v(t,x)$ with $|v|\in \mathcal{F}(0,+\infty)$, we can define its norm
\begin{equation*}
\tag{5.2}
||v||=\sup\limits_{t> 0, x\in V}\frac{|v(t,x)|}{\rho (t,x)},
\end{equation*}
where $\rho(t,x)=p(t+\gamma,e,x)$.

Let
\begin{equation*}
(\Phi u)(t,x)= \int_0^t\sum\limits_{y\in V}\mu(y)p(t-s,x,y)u(s,y)^{1+\alpha} ds.
\end{equation*}

We first prove some lemmas which are essential to prove the Theorem 3.2.


\begin{lemma}
\textnormal{If $G$ satisfies $CDE(n,0)$ and $V(x,r)\geq c_1 r^m$ with some positive constants $c_1$ and $m$.
Let $m\alpha>2$, then
$$\Phi \rho\in \mathcal{F}(0,+\infty) \quad \text{and} \quad ||\Phi\rho||\leq \widetilde{C},$$
where $\widetilde{C}$ is a positive constant.}
\end{lemma}


\begin{proof}[Proof]
For any $(t,x)\in(0,+\infty)\times V$,
\begin{equation*}
\begin{split}
(\Phi\rho)(t,x)&=\int_0^t\sum_{y\in V}\mu(y)p(t-s,x,y)\rho(s,y)^{1+\alpha} ds\\
&=\int_0^t\sum_{y\in V}\mu(y)p(t-s,x,y)p(s+\gamma,e,y)p(s+\gamma,e,y)^\alpha ds.
\end{split}
\end{equation*}

Obviously, $\Phi\rho$ is non-negative and continuous with respect to $t$.

According to Proposition 2.2, for $\gamma>0$ and any $s\geq0$, there exists a constant $C_1$
such that
$$p(s+\gamma,e,y)\leq \frac{C_1}{V\big(e,\sqrt{s+\gamma}\big)}.$$

Since
$$V\big(e,\sqrt{s+\gamma}\big)\geq c_1(s+\gamma)^{\frac{m}{2}},$$
we obtain
\begin{equation*}
\tag{5.3}
p(s+\gamma,e,y)\leq C_1c_1^{-1}(s+\gamma)^{-\frac{m}{2}}.
\end{equation*}

Hence,
\begin{equation*}
\begin{split}
(\Phi\rho)(t,x)&\leq\int_0^t\sum_{y\in V}\mu(y)p(t-s,x,y)p(s+\gamma,e,y)(C_1c_1^{-1})^\alpha(s+\gamma)^{-\frac{m\alpha}{2}} ds\\
&\leq(C_1 c_1^{-1})^\alpha\int_0^t(s+\gamma)^{-\frac{m\alpha}{2}} \sum_{y\in V}\mu(y)p(t-s,x,y)p(s+\gamma,e,y)ds\\
&=(C_1 c_1^{-1})^\alpha p(t+\gamma,e,x)\int_0^t(s+\gamma)^{-\frac{m\alpha}{2}}ds.
\end{split}
\end{equation*}

Furthermore,
\begin{equation*}
\tag{5.4}
\begin{split}
\int_0^t(s+\gamma)^{-\frac{m\alpha}{2}}ds&\leq\int_0^{+\infty} (s+\gamma)^{-\frac{m\alpha}{2}}ds\\
&=\frac{-2\gamma}{2-m\alpha}\gamma^{-\frac{m\alpha}{2}},
\end{split}
\end{equation*}
it is worth noting that the existence of the integral in (5.4) is based on the assumption $m\alpha>2$.

Thus for any $(t,x)\in(0,+\infty)\times V$,
\begin{equation*}
\tag{5.5}
(\Phi\rho)(t,x)\leq\widetilde{C}p(t+\gamma,e,x),
\end{equation*}
where $\widetilde{C}=\frac{-2\gamma}{2-m\alpha}(C_1 c_1^{-1})^\alpha \gamma^{-\frac{m\alpha}{2}}>0$.\\

It follows that
$$\Phi\rho\in \mathcal{F}(0,+\infty)\quad \text{and} \quad ||\Phi\rho||\leq\widetilde{C}.$$

This completes the proof of Lemma 5.1.
\end{proof}


\begin{lemma}
\textnormal{Under the condition of Lemma 5.1 and $u\in \mathcal{F}(0,+\infty)$, we have
$$\Phi u\in \mathcal{F}(0,+\infty) \quad \text{and} \quad ||\Phi u||\leq\widetilde{C}||u||^{1+\alpha}.$$}
\end{lemma}

\begin{proof}[Proof]
Since $u\in \mathcal{F}(0,+\infty)$, we can define its norm and then have
$u(t,x)\leq ||u||\rho(t,x)$ for any $(t,x)\in (0,+\infty)\times V$.

A simple calculations show that
\begin{equation*}
\tag{5.6}
\begin{split}
0\leq (\Phi u)(t,x)&=\int_0^t\sum\limits_{y\in V}\mu(y)p(t-s,x,y)u(s,y)^{1+\alpha} ds\\
&\leq ||u||^{1+\alpha}\int_0^t\sum\limits_{y\in V}\mu(y)p(t-s,x,y)\rho(s,y)^{1+\alpha} ds\\
&= ||u||^{1+\alpha}(\Phi\rho)(t,x).
\end{split}
\end{equation*}

Combining (5.6) with (5.5), we get
$$\Phi u\in \mathcal{F}(0,+\infty) \quad \text{and} \quad ||\Phi u||\leq\widetilde{C}||u||^{1+\alpha}.$$

This completes the proof of Lemma 5.2.
\end{proof}


\begin{lemma}
\textnormal{Under the condition of Lemma 5.1,
we suppose that $u$ and $v$ are in $\mathcal{F}(0,+\infty)$ and satisfy $||u||\leq M$ and $||v||\leq M$ with a positive number $M$. Then we have
$$||\Phi u-\Phi v||\leq\widetilde{C}(1+\alpha)M^\alpha||u-v||.$$}
\end{lemma}

\begin{proof}[Proof]
Since $u,v\in \mathcal{F}(0,+\infty)$, for any $(t,x)\in [0, \infty)\times V$, we get
$$\big|u(t,x)-v(t,x)\big| \leq |u(t,x)|-|v(t,x)| \leq 2M \rho(t,x),$$
which implies $|u-v|\in \mathcal{F}(0,+\infty)$.

By using of the elementary inequality
$$|p^{1+\alpha}-q^{1+\alpha}|\leq(1+\alpha)|p-q|\max\{p^\alpha, q^\alpha\} \quad\,\, (q\geq0, p\geq0),$$
we have
\begin{equation*}
\begin{split}
\big|u(s,y)^{1+\alpha}-v(s,y)^{1+\alpha}\big|&\leq
(1+\alpha)|u(s,y)-v(s,y)|\max\{u(s,y)^{\alpha},v(s,y)^{\alpha}\}\\
&\leq (1+\alpha)M^\alpha\rho(s,y)^\alpha\big|u(s,y)-v(s,y)\big|\\
&\leq (1+\alpha)M^\alpha\rho(s,y)^\alpha||u-v||\rho(s,y)\\
&=(1+\alpha)M^\alpha\rho(s,y)^{1+\alpha}||u-v||.
\end{split}
\end{equation*}

\textit{Case 1.}
When $G$ is a finite connected graph, for any $(t,x)\in (0,+\infty)\times V$, we find that
\begin{equation*}
\tag{5.7}
\begin{split}
&\big|\Phi u(t,x)-\Phi v(t,x)\big|\\
=&\Bigg|\int_0^t\sum\limits_{y\in V}\mu(y)p(t-s,x,y)\Big(u(s,y)^{1+\alpha}-v(s,y)^{1+\alpha}\Big) ds\Bigg|\\
\leq & \int_0^t\sum\limits_{y\in V}\mu(y)p(t-s,x,y)\Big|u(s,y)^{1+\alpha}-v(s,y)^{1+\alpha}\Big| ds\\
\leq & (1+\alpha)M^\alpha||u-v||\int_0^t\sum\limits_{y\in V}\mu(y)p(t-s,x,y)\rho(s,y)^{1+\alpha} ds\\
=&(1+\alpha)M^\alpha||u-v||(\Phi\rho)(t,x)\\
\leq & (1+\alpha)M^\alpha||u-v||\widetilde{C}\rho(t,x),
\end{split}
\end{equation*}
thus
\begin{equation*}
\tag{5.8}
||\Phi u-\Phi v||\leq\widetilde{C}(1+\alpha)M^\alpha||u-v||.
\end{equation*}

\textit{Case 2.}
When $G$ is a locally finite connected graph, we shall make an annotation on the above calculation.

Since $u\in \mathcal{F}(0,+\infty)$ and $||u||\leq M$, we have
$$0\leq u(t,x)\leq Mp(t+\gamma, e,x).$$

By (5.3), we know that
$$p(t+\gamma, e,x)\leq C_1c_1^{-1}\gamma^{-\frac{m}{2}},$$
hence for any $(t,x)\in (0,+\infty)\times V$, we deduce that
$$0\leq u(t,x)\leq A,$$
where $A=MC_1c_1^{-1}\gamma^{-\frac{m}{2}}$.

Similarly, $v$ also satisfies $0\leq v(t,x)\leq A$.

Hence, $\sum_{y\in V}\mu(y)p(t-s,x,y)u(s,y)^{1+\alpha}$ and
$\sum_{y\in V}\mu(y)p(t-s,x,y)v(s,y)^{1+\alpha}$ both are convergent, which shows that
\begin{equation*}
\begin{split}
&\sum\limits_{y\in V}\mu(y)p(t-s,x,y)u(s,y)^{1+\alpha}-\sum\limits_
{y\in V}\mu(y)p(t-s,x,y)v(s,y)^{1+\alpha}\\
=&\sum\limits_{y\in V}\mu(y)p(t-s,x,y)\Big(u(s,y)^{1+\alpha}-v(s,y)^{1+\alpha}\Big).
\end{split}
\end{equation*}

Based on the above discussion, we verify the validity of inequalities (5.7) and (5.8) under the condition that $G$ is locally finite.

The proof of Lemma 5.3 is complete.
\end{proof}


\begin{proof}[Proof of Theorem 3.2.]
(i) We construct the solution of (5.1) in $\mathcal{F}(0,+\infty)$.

Setting a iteration relation
\begin{equation*}
\tag{5.9}
u_{n+1}=u_0+\Phi u_n \quad \quad (n=0,1,\cdots)
\end{equation*}
with $u_0$ given by (5.1) and $u_n\in\mathcal{F}(0,+\infty)$, \,$(n=1,2,\cdots)$.

Since $0\leq a(y)\leq\delta p(\gamma,e,y)$, for any $(t,x)\in(0,+\infty)\times V$, we have
\begin{equation*}
\begin{split}
0\leq u_0(t,x)&\leq\delta\sum_{y\in V}\mu(y)p(t,x,y)p(\gamma,e,y)\\
&=\delta p(t+\gamma,e,x),
\end{split}
\end{equation*}
which shows $u_0\in\mathcal{F}(0,+\infty)$ and $||u_0||\leq\delta$.

According to Lemma 5.2, we obtain the inequalities
\begin{equation*}
\begin{split}
||u_{n+1}||\leq ||u_0||+||\Phi u_n||\leq ||u_0||+\widetilde{C}||u_n||^{1+\alpha},
\end{split}
\end{equation*}
that is,
\begin{equation*}
||u_{n+1}|| \leq \delta+\widetilde{C}||u_n||^{1+\alpha}, \quad \quad (n=0,1,\cdots).
\end{equation*}

It is easy to observe that
$$\lim_{\delta\rightarrow 0}\frac{\delta^{\frac{\alpha}{2}}\big(1+\delta^{\frac{\alpha}{4}}\big)^{1+\alpha}}{\delta^{\frac{\alpha}{4}}}=0,$$
so there exist some $\delta<1$ such that $\delta^{\frac{\alpha}{2}}\big(1+\delta^{\frac{\alpha}{4}}\big)^{1+\alpha}<\delta^{\frac{\alpha}{4}}$.

Setting $$\mathcal{A}\equiv \left\{ \delta: \; 0<\delta<1, \; \delta^{\frac{\alpha}{2}}\big(1+\delta^{\frac{\alpha}{4}}\big)^{1+\alpha}<\delta^{\frac{\alpha}{4}} \,\, \textrm{and} \,\, \widetilde{C}\delta^{\frac{\alpha}{2}}<1 \right\}.$$

For any $\delta\in\mathcal{A}$, we have
\begin{equation*}
\begin{split}
||u_0||&\leq\delta,\\
||u_1||&\leq\delta+\widetilde{C}\delta^{1+\alpha}<\delta+\delta^{1+\frac{\alpha}{2}},\\
||u_2||&\leq\delta+\widetilde{C}\big(\delta+\delta^{1+\frac{\alpha}{2}}\big)^{1+\alpha}\\
&<\delta+\delta^{1+\frac{\alpha}{2}}\big(1+\delta^{\frac{\alpha}{2}}\big)^{1+\alpha}\\
&<\delta+\delta^{1+\frac{\alpha}{2}}\big(1+\delta^{\frac{\alpha}{4}}\big)^{1+\alpha},\\
||u_3||&\leq\delta+\widetilde{C}\big[\delta+\delta^{1+\frac{\alpha}{2}}\big(1+\delta^{\frac{\alpha}{4}}\big)^{1+\alpha}\big]^{1+\alpha}\\
&<\delta+\delta^{1+\frac{\alpha}{2}}\big[1+\delta^{\frac{\alpha}{2}}\big(1+\delta^{\frac{\alpha}{4}}\big)^{1+\alpha}\big]^{1+\alpha}\\
&<\delta+\delta^{1+\frac{\alpha}{2}}\big(1+\delta^{\frac{\alpha}{4}}\big)^{1+\alpha},\\
&\cdots\\
||u_n||&<\delta+\delta^{1+\frac{\alpha}{2}}\big(1+\delta^{\frac{\alpha}{4}}\big)^{1+\alpha},\\
&\cdots
\end{split}
\end{equation*}
thus we can assume that $||u_n||\leq M$, $(n=0,1,\cdots)$ with a constant $M=M(\delta)$ satisfying
\begin{equation*}
\tag{5.10}
M(\delta)\rightarrow 0^+ \quad\, (\delta\rightarrow 0^+).
\end{equation*}

From (5.10), we can choose a constant $\delta\in\mathcal{A}$ such that
\begin{equation*}
\kappa\equiv \widetilde{C}(1+\alpha)M(\delta)^\alpha<1.
\end{equation*}

Note that
$$u_{n+2}-u_{n+1}=\Phi u_{n+1}-\Phi u_{n},$$
it follows Lemma 5.3 that
\begin{equation*}
\tag{5.11}
||u_{n+2}-u_{n+1}||\leq \kappa||u_{n+1}-u_n|| \quad\, (n=0,1,\cdots).
\end{equation*}

Since $\kappa<1$, the inequality (5.11) implies that $u_n$ converges with respect to the norm $||\cdot||$.

Moreover, for any Cauchy sequence $\{u_n\}$ in $\mathcal{F}(0,\infty)$,
one can easily conclude that $\{u_n\}$ is convergent.

Hence, there exists a function $u\in\mathcal{F}(0,+\infty)$ such that
\begin{equation*}
\tag{5.12}
||u_n-u||\rightarrow 0 \quad\, ( n\rightarrow \infty).
\end{equation*}

Utilizing (5.9) and (5.12) leads us to the assertion that $u$ is a solution of (5.1) in $\mathcal{F}(0,+\infty)$.\\

(ii) We prove that the solution $u(t,x)$ of (5.1) constructed above satisfies (3.1).

For any $T>0$, since $u\in \mathcal{F}(0,T]$, we derive from (5.3) that $u$ is bounded and continuous with respect to $t$ in $(0,T]\times V$.

Taking a small positive number $\varepsilon$, we put
\begin{equation*}
(\Phi_\varepsilon u)(t,x)=\int_0^{t-\varepsilon} \sum_{y\in V}\mu(y)p(t-s,x,y)u(s,y)^{1+\alpha} ds,
\end{equation*}
where $0<\varepsilon< t\leq T$ and $x\in V$.

Obviously, $\Phi_\varepsilon u$ tends to $\Phi u$ in $[\sigma,T]\times V$ as $\varepsilon\rightarrow 0^+$, here $\sigma$ is an arbitrary positive number and $\sigma>\varepsilon$.

\textit{Case 1.}
If $G$ is a finite connected graph, recalling an important property of heat kernel:
\begin{equation}
\tag{5.13}
p_t(t,x,y)=\Delta_x p(t,x,y)=\Delta_y p(t,x,y),
\end{equation}
we have
\begin{equation}
\tag{5.14}
\begin{split}
\frac{\partial}{\partial t}(\Phi_\varepsilon u)&=\sum_{y\in V}\mu(y)p(\varepsilon,x,y)u(t-\varepsilon,y)^{1+\alpha}+\int_0^{t-\varepsilon} \sum_{y\in V}
\mu(y)p_t(t-s,x,y)u(s,y)^{1+\alpha} ds\\
&=\sum_{y\in V}\mu(y)p(\varepsilon,x,y)u(t-\varepsilon,y)^{1+\alpha}+\int_0^{t-\varepsilon} \sum_{y\in V}
\mu(y)\Delta_x p(t-s,x,y)u(s,y)^{1+\alpha} ds\\
&\equiv I_1+I_2.
\end{split}
\end{equation}

Owing to the boundedness of $u^{1+\alpha}$, it is immediate from (4.4) that
$I_1$ tends to $u(t,x)^{1+\alpha}$ in $[\sigma,T]\times V$ as $\varepsilon\rightarrow 0^+$.
On the other hand, when $\varepsilon\rightarrow 0^+$, $I_2$ converges in $[\sigma,T]\times V$ to a function
\begin{equation*}
\varphi(t,x)=\int_0^{t} \sum_{y\in V}\mu(y) \Delta_x p(t-s,x,y)u(s,y)^{1+\alpha}ds.
\end{equation*}

Letting $\varepsilon\rightarrow 0^+$ in (5.14), we obtain for any $(t,x)\in[\sigma,T]\times V$,
\begin{equation*}
\tag{5.15}
\begin{split}
\frac{\partial}{\partial t}(\Phi u)(t,x)&=u(t,x)^{1+\alpha}+\varphi(t,x)\\
&=u(t,x)^{1+\alpha}+\Delta_x(\Phi u)(t,x).
\end{split}
\end{equation*}

Since $u=u_0+\Phi u$\, and \,$\frac{\partial}{\partial t}u_0=\Delta u_0$, for any $(t,x)\in[\sigma,T]\times V$,
we conclude that
\begin{equation*}
\tag{5.16}
\begin{split}
u_t&=\Delta u_0+\frac{\partial}{\partial t}\Phi u\\
&=\Delta u_0+\Delta(\Phi u)+u^{1+\alpha}\\
&=\Delta u+u^{1+\alpha}.
\end{split}
\end{equation*}

Because of the arbitrariness of $\sigma$, the (5.16) is true for all $(t,x)\in(0,T]\times V$.

Furthermore, we can prove that the initial-value condition is satisfied in the sense that
$$u(t,x)\rightarrow a(x) \quad\, (t\rightarrow 0^+),$$
from which we can extend $u(t,x)$ to $t=0$ and set $u(0,x)=a(x)$.

By the arbitrariness of $T$, we can deduce that the solution $u(t,x)$ of (5.1) constructed above is the required global solution of (3.1) in $\mathcal{F}[0,+\infty)$.

\textit{Case 2.}
If $G$ is a locally finite connected graph, as before, some facts need to be verified:

(a) \begin{equation*}
\frac{\partial}{\partial t}\left(\sum_{y\in V}\mu(y)p(t-s,x,y)u(s,y)^{1+\alpha}\right)
=\sum_{y\in V}\mu(y)p_t(t-s,x,y)u(s,y)^{1+\alpha};
\end{equation*}

(b) \begin{equation*}
\int_0^{t} \sum_{y\in V}\mu(y) \Delta_x p(t-s,x,y)u(s,y)^{1+\alpha}ds
=\Delta\left(\int_0^{t} \sum_{y\in V}\mu(y)  p(t-s,x,y)u(s,y)^{1+\alpha}ds\right).
\end{equation*}

Note that $u$ is bounded and $D_\mu<\infty$, we deduce that $\Delta u^{1+\alpha}$ is bounded too.

Following (4.5) and (5.13), we find that
\begin{equation*}
\begin{split}
\sum_{y\in V}\mu(y)p_t(t-s,x,y)u(s,y)^{1+\alpha} &=\sum_{y\in V}\mu(y)\Delta_yp(t-s,x,y)u(s,y)^{1+\alpha}\\
&=\sum_{y\in V}\mu(y)p(t-s,x,y)\Delta u(s,y)^{1+\alpha},
\end{split}
\end{equation*}
thus $\sum_{y\in V}\mu(y)p_t(t-s,x,y)u(s,y)^{1+\alpha}$ converges uniformly, from which we can see that the fact (a) is valid.

Similar to the proof of (4.5), the fact (b) is also true due to the absolute convergence of sums.

In view of (a) and (b), for a locally finite graph we will have the same conclusion as a finite graph.

This completes the proof of Theorem 3.2.
\end{proof}

\section{Proof of Corollary 3.1}

\begin{proof}[Proof of Corollary 3.1.]

As a direct consequence of Theorem 3.1, if we add a curvature condition $CDE'(n,0)$ to $G$, then we conclude that there is no global solution to problem (3.1) for the case of $0< m\alpha <2$

Actually, by Proposition 2.3, for any $t>\max\left\{\frac{1}{2}, \, 2r_0^2\right\}$, we have
\begin{equation*}
p(t,e,e)\geq c_0^{-1}C\left(\frac{t}{2}\right)^{-\frac{m}{2}}.
\end{equation*}

Hence
\begin{equation*}
J_0\geq \overline{C'}t^{-\frac{m}{2}},
\end{equation*}
where $\overline{C'}=2^{\frac{m}{2}}c_0^{-1}C\mu(e)a(e)>0$ and $t>\max\left\{\frac{1}{2}, \, 2r_0^2\right\}$.\\

Combining $J_0^{-\alpha}\geq \alpha t$ and $J_0\geq \overline{C'}t^{-\frac{m}{2}}$, for any $t>\max\left\{\frac{1}{2}, \, 2r_0^2\right\}$, we have
\begin{equation*}
\tag{6.1}
t^{\frac{m\alpha}{2}}\geq \alpha\overline{C'}^\alpha t.
\end{equation*}

However, if $0<m\alpha<2$, the above inequality (6.1) is
distinctly not true for a sufficiently large $t$. This proves the
assertion in the first part of Corollary 3.1.

On the other hand, Horn et al. \cite{HLLY} have concluded that $CDE'(n,0)$ implies $CDE(n,0)$. Thus
the assertion in the second part of Corollary 3.1 can be obtained
by replacing $CDE(n,0)$ with $CDE'(n,0)$ in Theorem 3.2.

This completes the proof of Corollary 3.1.
\end{proof}

\section{Example and numerical experiments}

In this section, we give an example to illustrate our result
asserted in Corollary 3.1

It is well known that the integer grid $\mathbb{Z}^m$ admits the uniform volume growth of degree $m$. In \cite{BHLLMY},
Bauer et al. proved that $\mathbb{Z}^m$ satisfies $CDE(2m,0)$ and $CDE'(4.53m,0)$ for the normalized graph Laplacian, from which we can deduce the existence and non-existence of global solution to problem (3.1) in $\mathbb{Z}^m$ with the normalized graph Laplacian, as follows:\\
\textbf{Proposition 7.1} \,
\textnormal{Let $G$ be $\mathbb{Z}^m$ with $\mu(x)\equiv \deg(x)$.\\
(i) If $0<m\alpha<2$, then there is no non-negative global solution of (3.1) in $[0,+\infty)$ for any bounded,
non-negative and non-trivial initial value.\\
(ii) If $m\alpha>2$, then there exists a global solution of (3.1) in $\mathcal{F}[0,+\infty)$ for a sufficiently small initial value. }

For example, we consider a circle $\mathrm{C}_\mathrm{n}$ (as shown in Figure 1) which satisfies $CDE'(n',0)$ for some number $n'$ related to $\mathrm{n}$. And then the problem (3.1) can be written as
\begin{equation*}
\tag{7.1}
\left\{
\begin{array}{lc}
u_t(t,x_1)=\frac{1}{2}\big(u(t,x_6)+u(t,x_2)\big)-u(t,x_1)+u(t,x_1)^{1+\alpha},\\
u_t(t,x_2)=\frac{1}{2}\big(u(t,x_1)+u(t,x_3)\big)-u(t,x_2)+u(t,x_2)^{1+\alpha},\\
u_t(t,x_3)=\frac{1}{2}\big(u(t,x_2)+u(t,x_4)\big)-u(t,x_3)+u(t,x_3)^{1+\alpha},\\
u_t(t,x_4)=\frac{1}{2}\big(u(t,x_3)+u(t,x_5)\big)-u(t,x_4)+u(t,x_4)^{1+\alpha},\\
u_t(t,x_5)=\frac{1}{2}\big(u(t,x_4)+u(t,x_6)\big)-u(t,x_5)+u(t,x_5)^{1+\alpha},\\
u_t(t,x_6)=\frac{1}{2}\big(u(t,x_5)+u(t,x_1)\big)-u(t,x_6)+u(t,x_6)^{1+\alpha},\\
u(0,x_1)=a(x_1),\\
u(0,x_2)=a(x_2),\\
u(0,x_3)=a(x_3),\\
u(0,x_4)=a(x_4),\\
u(0,x_5)=a(x_5),\\
u(0,x_6)=a(x_6),
\end{array}
\right.
\end{equation*}
where we take $\mu(x)=\sum_{y\sim x}\omega_{xy}=\deg(x)=2$.

\begin{figure}[!h]
\centering
\includegraphics[width=1.5in, height=1.5in, origin=br]{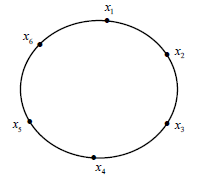}
\begin{minipage}{145mm}
\centering
\caption{$\mathrm{C}_\mathrm{n}$}
\label{fig.1}
\end{minipage}
\end{figure}

If we choose $\alpha=1, a(x_1)=1, a(x_2)=2, a(x_3)=3, a(x_4)=4, a(x_5)=5, a(x_6)=6$, respectively. It is easy to verify that the above choices satisfy the condition of non-existence of global solution to the equations (7.1). The numerical experiment result is shown in Figure 2.

Besides, if we choose $\alpha=3, a(x_1)=1\times 10^{-4}, a(x_2)=2\times 10^{-4}, a(x_3)=3\times 10^{-4}, a(x_4)=4\times 10^{-4}, a(x_5)=5\times 10^{-4}, a(x_6)=6\times 10^{-4}$, respectively. Then the above choices satisfy the condition of existence of global solution to the equations (7.1). The numerical experiment result is shown in Figure 3.

\begin{figure}[!h]
\centering
\includegraphics[width=4.5in, height=3in, origin=br]{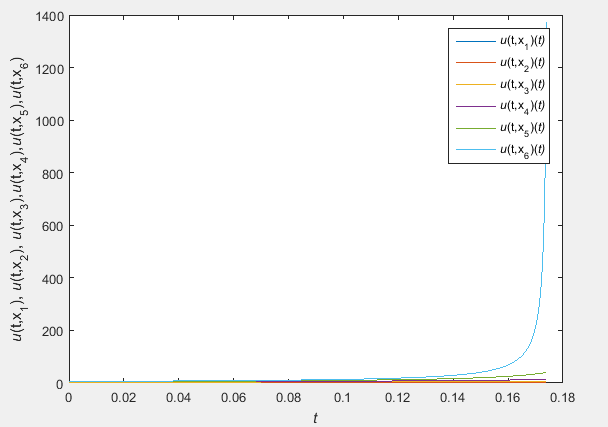}
\begin{minipage}{145mm}
\centering
\caption{Non-existence of global solution to the equations (7.1)}
\label{fig.2}
\end{minipage}
\end{figure}

\begin{figure}[!h]
\centering
\includegraphics[width=4.5in, height=3in, origin=br]{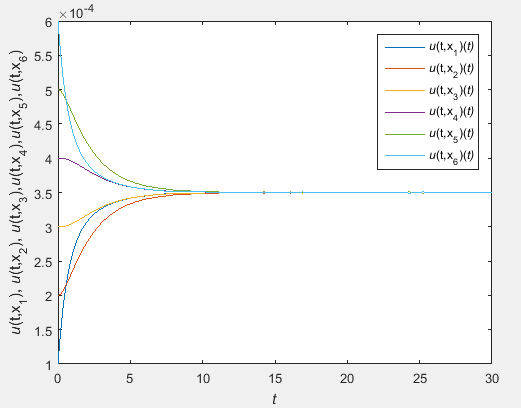}
\begin{minipage}{145mm}
\centering
\caption{Existence of global solution to the equations (7.1)}
\label{fig.3}
\end{minipage}
\end{figure}

\newpage

\section*{Acknowledgments}
This research is supported by the Fundamental Research Funds for the Central
Universities, and the Research Funds of Renmin University of China
under Grant 17XNH106.

\def\refname{\Large

\vspace{30pt}

{\setlength{\parindent}{0pt}
Yong Lin,\\
Department of Mathematics, Renmin University of China,  Beijing, 100872, P. R. China\\
linyong01@ruc.edu.cn\\
Yiting Wu,\\
Department of Mathematics, Renmin University of China,  Beijing, 100872, P. R. China\\
yitingly@126.com
                   }
\end{document}